\documentclass[a4paper, 12pt]{amsart}
\usepackage[utf8]{inputenc}
\usepackage{amsmath,amsthm,amsfonts,amssymb}
\usepackage{url}
\usepackage{color}
\usepackage{graphicx} 

\usepackage{mathtools} 

\usepackage{microtype}

\usepackage{enumitem}

\usepackage{tikz}

\usepackage{hyperref}
\hypersetup{colorlinks=true, linkcolor=black,          
 citecolor=black,        
}

\usepackage{tabularx}

\calclayout

\newcommand{\ot}{\otimes}

\newcommand{\Z}{{\mathbb{Z}}}

\newcommand{\R}{{\mathbb{R}}}
\newcommand{\C}{{\mathbb{C}}}

\newcommand{\cC}{\mathcal{C}}

\newcommand{\cJ}{\mathcal{J}}

\newcommand{\cCi}{\cC^\infty}

\newcommand{\al}{\alpha}
\newcommand{\be}{\beta}

\renewcommand{\phi}{\varphi}

\newcommand{\la}{\langle}
\newcommand{\ra}{\rangle}

\newcommand{\Diff}{\mathrm{Diff}}

\newcommand{\SSS}{\mathrm{S}}
\newcommand{\D}{\mathrm{D}}
\newcommand{\Spin}{\mathrm{Spin}}

\newcommand{\OO}{\mathrm{O}}

\newcommand{\supp}{\mathrm{supp}}

\DeclareMathOperator{\Jac}{Jac}

\DeclareMathOperator{\Ker}{ker}
\DeclareMathOperator{\End}{End}

\DeclareMathOperator{\Id}{Id}

\DeclareMathOperator{\Ima}{Im}
\DeclareMathOperator{\Rea}{Re}

\newtheorem{theorem}{Theorem}[section]
\newtheorem{corollary}[theorem]{Corollary}
\newtheorem{lemma}[theorem]{Lemma}
\newtheorem{proposition}[theorem]{Proposition}
\newtheorem{example}[theorem]{Example}

\newtheorem{innercustomthm}{Theorem}
\newenvironment{customthm}[1]
{\renewcommand\theinnercustomthm{#1}\innercustomthm}
{\endinnercustomthm}

\theoremstyle{definition}
\newtheorem{definition}[theorem]{Definition}
\theoremstyle{remark}

\newtheorem{remark}[theorem]{Remark}

\title[New geometric structures on $3$-manifolds]{New geometric structures  on $3$-manifolds:  \\ surgery and generalized geometry}
\author{Joan Porti and Roberto Rubio}

\thanks{This project has been supported by the European Union’s Horizon 2020 research and innovation programme under the Marie Sklodowska-Curie grant agreement No 750885 GENERALIZED and by the AGAUR under the grant 2021-SGR-01015. 
The first author has been partially supported by 
FEDER/AEI/MICIU through the grant PID2021-125625NB-100 and the
Mar\'\i a de Maeztu Program CEX2020-001084-M.
The second author has also received support from the FEDER/AEI/MICIU under the Ramón y Cajal fellowship RYC2020-030114-I and the grant PID2022-137667NA-I00.}

\begin{document}
\maketitle

\begin{abstract}
    Cosymplectic and normal almost contact structures are analogues of symplectic and complex structures that can be defined on $3$-manifolds. Their existence imposes strong topological constraints. Generalized geometry offers a natural common generalization of these two structures: $B_3$-generalized complex structures. We prove that any closed orientable $3$-manifold admits such a structure, which can be chosen to be stable, that is, generically cosymplectic up to generalized diffeomorphism. 
\end{abstract}

\section{Introduction}\label{sec:intro}

Generalized geometry encompasses classical geometric structures  (symplectic and complex structures are particular cases of generalized complex structures) but also offers genuinely new structures: there are closed manifolds that are neither complex nor symplectic but admit a stable generalized complex structure \cite{cavalcanti-gualtieri:2007, cavalcanti-gualtieri:2009}. This is a very tight phenomenon, as a generalized complex manifold must be almost complex. The key for this small window to exist is the type change, which happens from dimension four onwards. A stable generalized complex manifold is generically symplectic but degenerates in a complex way along a codimension-$2$ submanifold. 

Generalized geometry of type $B_n$ is an extension of generalized geometry that offers a whole new realm of structures. $B_n$-generalized complex structures are now possible in odd dimensions (having as particular cases cosymplectic and normal almost contact structures). Type change for stable $B_n$-generalized complex structures or simply $B_n$-structures, which are generically cosymplectic up to generalized diffeomorphism, occurs again in a codimension-$2$ submanifold and is already possible in dimension three. A natural question arises: does there exist a $B_3$-generalized complex manifold that is neither cosymplectic nor normal almost contact? 

We settle this question by proving the following general result.

\begin{customthm}{\ref{thm:every-closed-B3}}    Any closed orientable $3$-manifold admits a $B_3$-structure with exactly two type-change curves.
\end{customthm}

The proof combines mainly two ideas. On the one hand, the introduction of a family of surgeries on a cosymplectic neighbourhood which always results in a twisted $B_3$-structure (a larger class which is key to this work) together with a way to combine surgeries to cancel their twisting. On the other hand, a way of expressing any $3$-manifold as the result of a surgery on a mapping torus of an area-preserving diffeomorphism that fixes an open set. This is a combination of the open-book decomposition, Moser's argument, the transitivity of symplectomorphisms and the Dacorogna-Moser theorem. Technicalities aside, the way to prove Theorem \ref{thm:every-closed-B3} is to endow the corresponding mapping torus with a cosymplectic structure and perform two parallel surgeries, which cancel each other's twisting and whose effect on the diffeomorphism type is that of one surgery. We thus recover the original manifold with a $B_3$-structure. 
In a forthcoming work we will study the geometry and invariants associated to $B_3$-structures.

An extra lesson to be drawn is that, in the case of odd-dimensions, $B_n$-generalized complex structures, expressed as a differential form (of possibly mixed degree) are somehow simpler than their involved classical counterparts (a normal almost contact structure consists of an endomorphism, a vector field and a $1$-form satisfying compatibility and integrability conditions). 

We begin in Section \ref{sec:Bn-gen-cplx} with an account of generalized and $B_n$-generalized complex geometry from the viewpoint of differential forms. As far as we know, this is a novel way of presenting the subject, first used in the talk \cite{talk-poisson}. In Section \ref{sec:surgery} we introduce the family of generalized surgeries that we will use. With these we produce the first interesting examples, which we describe in Section \ref{sec:examples}. The proof of Theorem \ref{thm:every-closed-B3} is in Section~\ref{sec:every-closed-B3}.

\bigskip
\noindent \textbf{Notation: } We work on the smooth category. For simplicity, when working over a manifold $M$ we use $\Omega^k$ for the $k$-forms and $\Omega^k_{cl}$ for the closed $k$-forms, which may be defined locally.   
When a differential form, say $\rho$, is nowhere vanishing, we use the notation $\rho\neq 0$.


\bigskip
\noindent \textbf{Acknowledgements: } R.R. would like to thank Nigel Hitchin for suggesting the study of $3$-manifolds, Marco Gualtieri for helpful comments, and the Weizmann Institute of Science for its hospitality.


    




\section{$B_n$-generalized complex structures}
\label{sec:Bn-gen-cplx}

This section is a short introduction to $B_n$-generalized complex structures. For the standard viewpoint and more details, check \cite{rubio:2014} or, for usual generalized geometry, \cite{gualtieri-PhD:2004, gualtieri-annals:2011}.


\subsection{Generalized complex geometry through differential forms}\label{sec:gengeo-through-forms}

Many geometric structures are or can be described through differential forms, for instance, symplectic and complex structures. Consider a $2m$-dimensional manifold $M$. A symplectic structure is a $2$-form $\omega$ that is non-degenerate ($\omega^m\neq 0$) and closed ($d\omega=0$). On the other hand, a complex structure $J$ on $M$ is equivalently described by a subbundle $K_J\subset \wedge^m T^*_\C M$ admitting non-vanishing local sections $\zeta$ that are decomposable ($\zeta=\theta_1\wedge\ldots\wedge \theta_m$ for some local complex $1$-forms $\theta_j$), satisfy a non-degeneracy condition ($\zeta\wedge \overline{\zeta}\neq 0$) and are closed ($d\zeta=0$). Note that $\zeta$ plays the role of an $(m,0)$-form, whose annihilator is precisely the $-i$-eigenbundle of~$J$.

Both symplectic and complex structures can be put on the same footing by taking advantage of the fact that $\wedge^\bullet T^*_\C M$ is a Clifford module for $T_\C M\oplus T^*_\C M$, which we recall briefly. We fix the notation $X+\al$ (or $Y+\be$) and $\rho$ (or $\psi$) for the elements or local/global sections of $T_\C M\oplus T^*_\C M$  and $\wedge^\bullet T^*_\C M$, respectively. Pointwise, the natural pairing on $T_\C M\oplus T^*_\C M$ is determined by $\la X+\al,X+\al\ra=\alpha(X)$ and the Clifford action is
$$(X+\al)\cdot \rho := \iota_X \rho + \al\wedge \rho.$$
Since this is essentially the spinor representation, we refer to $\rho$ (or $\psi$) as spinors. The annihilator of a spinor is isotropic  for $\la\, , \ra$. The spinors having a maximally isotropic annihilator are called pure and are, pointwise, a nonzero multiple of the form 
 \begin{equation}
     \label{eq:pure}  \rho=e^{B+i\omega}\theta_1 \wedge \ldots \wedge \theta_k
 \end{equation}
for some $B$, $\omega\in \wedge^2 T^* M$ and $\theta_j\in\wedge^1 T^*_\C M$, where $k=0$ is possible and by the action of $e^{B+i\omega}$ on any $\psi$ we mean 
\begin{equation}\label{eq:B-action}
e^{B+i\omega}\psi:=e^{B+i\omega}\wedge \psi.    
\end{equation}

Motivated by this, for a symplectic form $\omega$, we consider the differential form $$e^{i\omega}:=1+i\omega-\omega^2/2!+\ldots,$$ which (globally) generates the subbundle $K_\omega :=\la e^{i\omega} \ra$. We thus have that both $K_\omega,K_J \subset \wedge^\bullet T^*_\C M$ and that both $e^{i\omega}$ and any section of $K_J$ are pure spinors.

Moreover, spinors come with an $\Spin_0$-invariant pairing with values in $\wedge^{top} T^*_\C M$ given by 
\begin{equation} \label{eq:pairing}
    ( \rho, \psi ) := [\rho^\top \wedge \psi]_{top},
\end{equation} 
 where $\phantom{\cdot}^\top$ denotes the reversing antiautomorphism $(\alpha_1\wedge \ldots \wedge \alpha_r)^\top:=\alpha_r\wedge\ldots\wedge \alpha_1$ for $1$-forms $\alpha_j$, and $[\;\,]_{top}$ the top-degree component. With respect to this pairing, we have that for any section $\zeta$ of $K_J$, $(\zeta,\overline{\zeta})\neq 0$, and $(e^{i\omega},e^{-i\omega})$ is a nonzero multiple of $\omega^m\neq 0$.
 
Finally, since $d\omega=0$ is equivalent to $de^{i\omega}=0$, we have that both $e^{i\omega}$ and any section of $K_J$ are closed. However, we consider, to our benefit, a weaker integrability condition that involves the Clifford action: we say that $\rho$ is integrable if there exists some complex $X+\al$ such that
\begin{equation*}\label{eq:drho=X+alrho}
d\rho = (X+\al)\cdot \rho.    
\end{equation*}

Generalized complex structures are the `greatest common divisor of symplectic and complex structures seen through the Clifford algebra glasses'.

\begin{definition}\label{def:gen-cplx}
    A generalized complex structure on a manifold $M$ is a subbundle $K\subset \wedge^\bullet T^*_\C M$ admitting local sections $\rho$ such that:
\begin{itemize}
    \item $\rho$ is pointwise pure, that is, $\rho=e^{B+i\omega}\theta_1\wedge \ldots \wedge \theta_k$ as in \eqref{eq:pure}.
    \item $\rho$ has real index zero, that is, $(\rho,\overline{\rho})\neq 0$ for the pairing in \eqref{eq:pairing}.
    \item $\rho$ is integrable, that is, $d\rho = (X+\al)\cdot \rho$ for some complex section $X+\alpha$. 
\end{itemize}
We call $k$ the type, which is an upper-semicontinuous integer function bounded by $m$.
\end{definition}



Note that purity implies constant parity on $\rho$, which together with $(\rho,\overline{\rho})\neq 0$ forces the dimension of $M$ to be even. The existence of a generalized complex structure implies the existence of an almost complex structure \cite[Sec. 3]{gualtieri-annals:2011}. So generalized complex geometry interacts with the border between complex and almost complex structures: in fact, there are generalized complex compact $4$-manifolds, like $3\mathbb{C}P^2$, that are neither complex nor symplectic \cite{cavalcanti-gualtieri:2009}. 

\begin{remark}\label{rem:formalism-Courant}
    Our approach recovers the usual presentation of generalized complex geometry in the same way that $J$ was recovered by $\zeta$ above. The annihilator of $\rho$ does not depend on the choice of section of $K$ and gives a subbundle $L\subset T_\C M \oplus T^*_\C M$ such that 
    \begin{itemize}
    \item $L$ is maximally isotropic, because of the purity of $\rho$.
    \item $L\cap \overline{L}=\{0\}$, because of $(\rho,\overline{\rho})\neq 0$.
    \item $L$ is involutive for the Dorfman bracket on $\Gamma(T_\C M \oplus T^*_\C M)$,
    $$[X+\al,Y+\be]:=[X,Y]+L_X\be -\iota_Y d\al.$$
    \end{itemize}
    The latter condition is equivalent to $d\rho=(X+\al)\cdot \rho$ for some complex $X+\al$, which explains the weaker integrability condition we can afford. The subbundle $L$ then plays the role of the $+i$-eigenbundle of a skew-symmetric endomorphism $\cJ$ satisfying $\cJ^2=-\Id$ \cite[Def. 3.1]{gualtieri-annals:2011}.
\end{remark}

\subsection{$B_n$-generalized complex geometry through differential forms}

For $M$ an $n$-dimensional manifold, $B_n$-generalized complex geometry is an extension of generalized complex geometry based on the fact that $\wedge^\bullet T^*_\C M$ is a Clifford module also for $T_\C M\oplus 1_\C \oplus T^*_\C M$ (where  $1_\C:=M\times \C$, and whose elements or local/global sections we denote by $X+f+\al$) with the pairing
\begin{equation}\label{eq:Bn-pairing}
\la X+f+\al,X+f+\al \ra := i_X \al + f^2
\end{equation}
and the Clifford action 
\begin{equation*}\label{eq:spinorial-action}
 (X+f+\al)\cdot \rho := \iota_X \rho + f\tau\rho +\xi\wedge \rho, \end{equation*}
where $\tau\rho=\tau(\rho_+ + \rho_-):=\rho_+ - \rho_- $ for $\rho_{\pm}$ the even/odd components of $\rho$. Being pure now corresponds, pointwise, to being a nonzero multiple of \begin{equation}\label{eq:pure-Bn} 
    \rho= e^{A+i\sigma} e^{B+i\omega} \theta_1\wedge \ldots \wedge \theta_k,
\end{equation}
for $A,\sigma\in\wedge^1 T^*M$, $B,\omega\in\wedge^2 T^*M$ and $\theta_j\in \wedge^1 T^*M_\C$, with $e^{B+i\omega}$ acting as in \eqref{eq:B-action} and  $e^{A+i\sigma}$ acting, on an arbitrary $\psi$, by
 \begin{equation}\label{eq:A-field}
 e^{A+i\sigma} \psi := \psi + (A+i\sigma)\wedge \tau\psi. \end{equation}
 The pairing on spinors is \eqref{eq:pairing} for $n$ even and, for $n$ odd, it is given by
\begin{equation}
    \label{eq:pairing-Bn} 
( \rho, \psi ):=[\widetilde{\rho}\wedge \psi]_{top},
\end{equation}
where $\widetilde{\rho}:=(\tau \rho)^\perp=\rho_+^\perp-\rho_-^\perp$ is the Clifford conjugation.


\begin{remark}
    Since the pairing \eqref{eq:Bn-pairing} is of signature $(n+1,n)$, which is preserved by the Lie group $\OO(n+1,n)$, a real form of $\OO(2n+1,\C)$ (of Lie type $B_n$), the prefix $B_n$ is used. 
\end{remark}

By analogy with Definition \ref{def:gen-cplx}, we introduce the relevant structures for this work.

\begin{definition}\label{def:Bn-gcs}
    A $B_n$-generalized complex structure on $M$ is a subbundle $K\subset \wedge^\bullet T^*_\C M$ admitting local sections $\rho$ such that:
\begin{itemize}
    \item $\rho$ is pointwise pure, that is, $\rho=e^{A+i\sigma} e^{B+i\omega}\theta_1\wedge \ldots \wedge \theta_k$ as in \eqref{eq:pure-Bn}.
    \item $\rho$ has real index zero, that is, $(\rho,\overline{\rho})\neq 0$ for the pairing \eqref{eq:pairing-Bn}/\eqref{eq:pairing} for $n$ odd/even.
    \item $\rho$ is integrable, that is, $d\rho = (X+f+\al)\cdot \rho$ for some complex section $X+f+\alpha$. 
\end{itemize}
We call $k$ the type, which is an upper-semicontinuous integer function bounded by $\lfloor n/2 \rfloor$.
\end{definition}

Note that any generalized complex structure is, in particular, a $B_n$-generalized complex structure. However, $B_n$-generalized complex structures are much more flexible. To start with, they can be defined also in odd dimensions $n=2m+1$. To give examples, we first recall the analogues of symplectic and complex structures in odd dimensions.

\begin{definition}\label{def:cosymplectic}
 A cosymplectic structure is a pair $(\sigma,\omega)$ consisting of  global forms $\omega\in\Omega^2_{cl}(M)$ and $\sigma\in\Omega^1_{cl}(M)$ with $\sigma\wedge\omega^m \neq 0$.
 \end{definition}
  
\begin{definition}
An almost contact structure is a tuple $(Y,\eta,J)$ consisting of a vector field $Y$, a $1$-form $\eta$ and $J\in \End(TM)$ satisfying $\iota_Y\eta = 1$ and $J^2=-\Id + Y\ot \eta$ (as a consequence, $J(Y)=0$,  $\eta\circ J=0$ and $J$ defines a complex structure on the distribution $\Ker \eta$).
It is said to be normal (nacs) when the almost complex structure $\widetilde{J}$ defined on the manifold $M\times \R$ as the extension of $J_{|\Ker \eta}$ by $\widetilde{J}Y=\partial_r$, with $r$ the $\R$-coordinate, is integrable.
\end{definition}

\begin{example}\label{ex:spinors-cosym-nacs}
A cosymplectic structure determines the $B_n$-generalized complex structure globally given by $e^{i(\sigma+\omega)}$, whereas a nacs determines one, for any choice of an $(m,0)$-form $\Omega$ on the annihilator  $Y^\circ$, given by $\rho=\Omega+i  \Omega\wedge \eta.$
\end{example}

\begin{example}\label{ex:S2xS1}
In particular, the cosymplectic structure on $\SSS^2\times \SSS^1$ (or $\D^2\times \SSS^1$) given by $(d\theta, \omega)$, with $d\theta$ the angle form of $\SSS^1$ and $\omega$ the area form of $\SSS^2$ (or $\D^2$) determines the spinor
$$ 1 + i d\theta + i \omega - d\theta \wedge \omega. $$
\end{example}

The examples most relevant to us are those with type change.

\begin{example}\label{ex:type-change-CxR}
On $\C\times \R$ with complex coordinate $z$ and real coordinate $t$,  the form $\rho=z+dz+idz\wedge dt$ defines a $B_3$-generalized complex structure (note that $d\rho=i\partial_t\cdot \rho$). Its type is $0$ along $\{z\neq 0\}$ and $1$ elsewhere.
\end{example}

As we see, in practice, instead of giving $K$, we will consider a family of local differential forms $\rho$ that, on the intersections, differ by a non-vanishing complex function.

\subsection{Symmetries: $A$ and $B$-fields}\label{sec:symmetries} Let $\rho$ be a $B_n$-generalized complex structure (possibly defined locally). For a diffeomorphism $f\in\Diff(M)$, the pullback $f^* \rho$ is again $B_n$-generalized complex.

We also look at the actions of $e^{A}$ and $e^B$ as in \eqref{eq:B-action} and \eqref{eq:A-field}, but for real forms. Note that, for $A, A'\in \Omega^1$, $B, B'\in \Omega^2$, the action of $e^B$ commutes both with that of $e^{B'}$ and $e^A$, whereas, acting as endomorphisms, 
\begin{equation}\label{eq:eAeA'}
e^{A}e^{A'}=e^{A+A'}e^{-A\wedge A'}.    
\end{equation}
Hence, the actions of $e^{A}$ and $e^B$ preserve the purity condition. Also,  
\begin{equation}\label{eq:eAeB-pairing}
    (e^{A}e^B\rho,\overline{e^{A}e^B\rho})=(e^{A}e^B\rho,e^{A}e^B\overline{\rho})=(\rho,\overline{\rho}).
\end{equation} 
Finally, we have
\begin{equation}
\label{eq:deBrho}
\begin{split}
d(e^B\rho) & = dB\wedge e^B\rho + e^Bd\rho,\\ d(e^A\rho) & = dA\wedge \tau \rho + e^A d\rho \\ & = dA\wedge \tau (e^A \rho) + dA\wedge A \wedge (e^A \rho) + e^A d\rho, 
 \end{split}\end{equation} 
which we shall combine with the following result.
\begin{lemma}\label{lem:eBeA}
    For $v:=X+f+\al$, a spinor $\rho$, and forms $B\in \Omega^2$, $A\in\Omega^1$, we have
\begin{align*}
e^B(v\cdot \rho) &= (e^{-B}v)\cdot (e^B\rho), & 
e^A(v\cdot \rho) &= (e^{-A}v)\cdot (e^A\rho),
\end{align*}
where, by definition,
\begin{align*}
    e^{-B}v&=e^{-B}(X+f+\alpha):=X+f+\alpha-\iota_X B,\\
    e^{-A}v&=e^{-A}(X+f+\alpha):=X+f-\iota_X A + \alpha + 2fA.
\end{align*}
\end{lemma}
\begin{proof}
    It can be checked directly (see also  \cite[Ch. 1]{rubio:2014}).
\end{proof}

\begin{remark}
Identity \eqref{eq:eAeB-pairing} and Lemma \ref{lem:eBeA} correspond to the fact that $\exp(A)$ and $\exp(B)$ are elements of $\Spin_0$ acting as $e^{\pm A}$ or $e^{\pm B}$ depending on whether they act on $v$ or $\rho$.
\end{remark}
\noindent Thus, by \eqref{eq:deBrho} and Lemma \ref{lem:eBeA}, we have that $e^Ae^B\rho$ is integrable for all $\rho$ integrable if and only if $dA=0$, $dB=0$. For global $A$ and $B$ these are new symmetries joining the theory.

\begin{definition}
    An $A$-field is an endomorphism $e^A$, for $A\in \Omega^1_{cl}(M)$, acting on $\Omega^\bullet_\C$ as in \eqref{eq:A-field}, and a $B$-field is an endomorphism $e^B$, for $B\in\Omega^2_{cl}(M)$, acting on $\Omega^\bullet_\C$ as in \eqref{eq:B-action}. 
\end{definition}

We have proved in this section that diffeomorphisms, $A$ and $B$-fields generate a group acting on $B_n$-generalized complex structures. This group is called the group of generalized diffeomorphisms.  For instance, any type $0$ structure is equivalent by a generalized diffeomorphism to a cosymplectic structure as in Example \ref{ex:spinors-cosym-nacs}.


\subsection{The integrability condition revisited}
\label{sec:integrability-revisited}

 Whereas in Examples \ref{ex:spinors-cosym-nacs} and \ref{ex:S2xS1} the forms are closed, in Example \ref{ex:type-change-CxR} it is not the case. However, wherever $\rho_0\neq 0$ we can always find a closed representative.  

\begin{lemma}\label{lem:rho/rho0-closed}
Consider a $B_n$-generalized complex structure locally given by some $\rho$. Wherever $\rho_0\neq 0$, the equivalent description $1+\rho_1/\rho_0 + \ldots + \rho_n/\rho_0$ is a closed form.
\end{lemma}

\begin{proof}
	By purity, the equivalent description can be written as $\rho/\rho_0=e^{A+i\sigma}e^{B+i\omega}$. The integrability condition for some $X+f+\alpha$ is 
	\[
	(dA+id\sigma + dB+i\omega) \wedge \rho /\rho_0=(\iota_X A + i \iota_X \sigma + \iota_X B  +  i \iota_X \omega + \alpha)\wedge  \rho/\rho_0 + f\tau(\rho/\rho_0).
	\]
	By looking at degree $0$ and using the condition in degree $1$, we get
    \begin{align*}
       0&= \iota_XA+i\iota_X \sigma + f,& 0&=\iota_X B+i\iota_X \omega + \alpha - 2f(A+i\sigma).
    \end{align*}
    Using these expressions in degrees $2$ and degrees $3$, we get, respectively,
    \begin{align*}
        dA+id\sigma&=0,&   dB+id\omega&=0,
    \end{align*}
    so the left-hand side is clearly zero, that is, $d(\rho/\rho_0)=0$.
\end{proof}
For example, wherever $z\neq 0$, the spinor $1+\frac{dz}{z}+i\frac{dz}{z}\wedge dt$ coming from Example \ref{ex:type-change-CxR} is indeed closed.

\begin{lemma}
    A necessary condition for the integrability of a $B_n$-generalized complex structure $\rho=\rho_0+\ldots+\rho_n$ is that, for all degrees $j$,
    \begin{equation}\label{eq:equality-by-components}
        \rho_0 d\rho_j=d\rho_0\wedge \rho_j.
    \end{equation}
\end{lemma}

\begin{proof}
    Let $p\in M$. When $\rho(p)\neq 0$, condition \eqref{eq:equality-by-components} follows from Lemma \ref{lem:rho/rho0-closed}. If $\rho_0(p)=0$ and $d\rho_0(p)=0$, it is trivially satisfied, and when $\rho_0(p)=0$ and $d\rho_0(p)\neq 0$, then $p$ is an isolated zero of $\rho_0$ and it follows from the case $\rho(p)\neq 0$ by continuity.
\end{proof}

\subsection{Twisted structures}

Although our focus is on $B_n$-generalized complex structures, we will need to consider the larger class of twisted structures, for which we replace the integrability in Definition \ref{def:Bn-gcs} by:
\begin{itemize}
    \item twisted integrable: there exist a $3$-form $H\in\Omega^3(M)$ and a closed $2$-form $F\in\Omega^2_{cl}(M)$ satisfying $dH+F^2=0$ such that, locally for a spinor $\rho$, there exists some complex $X+f+\alpha$ satisfying   
$$(d+ F\wedge \tau + H\wedge)\rho = (X+f+\al)\cdot \rho.$$
\end{itemize}
Note that $H$ and $F$ are global whereas $\rho$ and $X+f+\alpha$ can be local.

We see next that acting on $\rho$ by $e^A$ or $e^B$, for $A$, $B$ not necessarily closed, twists the integrability condition and justifies the definition. 

\begin{lemma}\label{lem:twist-int}
    For a twisted $B_n$-generalized complex structure (with respect to $F$ and $H$) locally given by $\rho$, and forms $A\in\Omega^1$ and $B\in\Omega^2$, we have that $e^A e^B\rho$ is twisted integrable with respect to
$$ d + (F-dA)\wedge \tau + (H-dB-(dA-2F)\wedge A)\wedge.$$
\end{lemma}


\begin{proof}
It follows from \eqref{eq:deBrho}, Lemma \ref{lem:eBeA} and the properties that $e^B$ commutes with $F\wedge \tau$ and $H$,  $e^{A}$ commutes with $H$, and 
\[
F\wedge \tau(A\wedge \tau\rho)=-F\wedge A\wedge \rho=  - A\wedge \tau(F\wedge \tau\rho). \] 

\vspace{-.3cm} \end{proof}

\subsection{Stable structures}

For a $B_n$-generalized complex structure given by $K$, any non-vanishing local section $\rho\in \Gamma(K)$ gives via $\rho_0$ a complex function on an open set. The preimage $\rho_0^{-1}(0)$ is independent of the choice of $\rho$.

\begin{lemma}\label{lem:rho0-cplx-coord}
    If $\rho_0$ vanishes transversely at $p$ (that is, $\rho_0(p)=0$ and $d\rho_0(p)\neq 0$), then, around $p$, the set $\rho_0^{-1}(0)$ is a codimension-$2$ submanifold and $\rho_0$ defines a complex coordinate transverse to it.
\end{lemma}

\begin{proof}
   From \eqref{eq:equality-by-components}, we have that $d\rho_0(p)\wedge \rho_j(p)=0$, so $\rho_j(p)=d\rho_0(p)\wedge \alpha_j$ for some $j$-form $\alpha_j$ at $p$. Consequently $(\rho,\overline{\rho})(p)=d\rho_0(p)\wedge \overline{d\rho_0}(p)\wedge \beta$ for some $(n-2)$-form $\beta$ at $p$. From the condition $(\rho,\overline{\rho})\neq 0$, it follows that $d\rho_0(p)\wedge \overline{d\rho_0}(p)\neq 0$, from which we deduce $\Rea d\rho_0(p) \wedge \Ima d\rho_0 (p)\neq 0$. Hence, $\rho_0^{-1}(0)$ is a codimension-$2$ submanifold (by the regular value theorem)  and, around $p$, the function $\rho_0$ defines a transverse complex coordinate. 
\end{proof}

The transverse vanishing of any local function $\rho_0$ is equivalently expressed as the transverse vanishing of the (global) anticanonical section $s\in \Gamma(K^*)$ induced by the projection from $K$ to its degree-$0$ component. The first part of Lemma \ref{lem:rho0-cplx-coord} motivates the following definition.

\begin{definition}
    We say that a $B_n$-generalized complex structure given by $K$ is stable when the anticanonical section $s\in \Gamma(K^*)$ vanishes transversely. For simplicity, we refer to these structures as $B_n$-structures.
\end{definition}

For a $B_n$-structure, the set $s^{-1}(0)$ is a codimension-$2$ manifold. For $B_3$-structures, where the type only takes values $0$ and $1$, this is exactly the type-change locus. 

From now on, we will focus on the class of stable structures. It is generic, has been the main focus of standard generalized geometry (see, for instance, the recent works \cite{cavalcanti-gualtieri:2018, cavalcanti-klaase}) and is connected to $\log$ or $b$-symplectic geometry \cite{gmp-2014}. Note that they are generically of type $0$.



\subsection{The main question}

Generalized structures usually encompass classical structures and a natural question is whether there are manifolds admitting the generalized structure but not the classical ones, as we mentioned for generalized complex structures in Section~\ref{sec:gengeo-through-forms}.

Note that a manifold admitting a $B_n$-generalized complex structures must be orientable, as the nowhere vanishing top form $(\rho,\overline{\rho})$ is defined up to a positive multiple. Type change is already possible for dimension $3$, as in Example \ref{ex:type-change-CxR}, so it is natural to ask \textbf{whether there are closed orientable $3$-manifolds that are are neither cosymplectic nor nacs but admit a $B_3$-structure}.

\begin{remark}\label{rem:cosym-nacs}
    Note that not every orientable $3$-manifold is cosymplectic or nacs. Cosymplectic manifolds are symplectic mapping tori (as explained, for instance, in \cite{li-cosymplectic}), whereas nacs closed $3$-manifolds were characterized by Geiges \cite{geiges-normal} and, by comparing with \cite{scott}, are equivalently described as orientable Seifert manifolds with orientable base.
\end{remark} 

In order to provide an answer to the question above, we first develop a family of surgeries.

\section{Surgeries on a cosymplectic neighbourhood}\label{sec:surgery}

Surgeries in four-dimensional generalized complex geometry combine a Dehn filling with a $\cCi$-log transform in a structure-compatible way \cite{cavalcanti-gualtieri:2007,cavalcanti-gualtieri:2009, torres:2012, goto-hayano}. In three dimensions, however, we have one less `degree of freedom' (one less variable) and one more degree component to take care of (we have both even and odd degrees). Still, we are able to find an analogous class of surgeries that will serve our purposes. We first fix some notation. All the disks will be $2$-dimensional disks. For $a > b > 0$, we denote by $D_a$ the open disk of radius $a$, and use the notation $R(a,b):=D_a\setminus \overline{D_b}$ for open rings. From now on, for the basics on $3$-manifold topology we refer to \cite{martelli2022introduction}.

\subsection{The $(p,\pm 1)$-surgery}\label{sec:p-1-surgery}


Start with a closed manifold $M$ endowed with a $B_3$-structure. Assume there is a neighbourhood where the structure is cosymplectic, given by $\sigma$ and $\omega$ (as in Example \ref{ex:spinors-cosym-nacs}), containing a  knot $\imath\colon \SSS^1\to M$ such that $\imath(\SSS^1)$ is transverse to $\ker \sigma$. Restrict this neighbourhood to a tubular neighbourhood $N\to \SSS^1$ in which the area of each fibre $\jmath\colon L\to N$ is constant, $\int_L \jmath^*\omega=\pi a^2$ for some $a>0$. Let $c\in \R^\times$ be such that $\int_{\SSS^1} \imath^* \sigma=2\pi c$. By Moser's trick we have a diffeomorphism of 
cosymplectic manifolds 
\begin{equation}\label{eq:diffeo-N-D-S1}
    (N,(\sigma,\omega))\cong (D_a\times \SSS^1,(cdt, sds\wedge d\phi)), 
\end{equation}
where $sds\wedge d\phi$ denotes the area form of $D_a$, written in polar coordinates $(s,\phi)$, and $t$ is the angular coordinate for $\SSS^1$. We use this diffeomorphism as an identification from now on. The $B_3$-structure is given, as in Example \ref{ex:S2xS1}, by
 \begin{equation}\label{eq:rho-cosymplectic-nhbd1}
 \rho:=1 + i c dt + i sds\wedge d\phi - cdt \wedge sds\wedge d\phi.	
 \end{equation}
Restrict $\rho$ to a ring $R_0:=R(a,b)$ times $\SSS^1$ for some $b$ with $a>b>0$, and pull it back to the ring $R_1:=R(e^{a^2/2},e^{b^2/2})$ times $\SSS^1$ via the map
\begin{alignat}{4} 
\Psi: & {R_1} \times \SSS^1  & \longrightarrow & R_0 \times \SSS^1 \nonumber \\
 & (r,\theta,\sigma) & \longmapsto & (\sqrt{ 2\log r}, p\theta -q\sigma, q\theta)=(s,\phi,t), \label{eq:map-psi}
\end{alignat}
 for some arbitrary integer $p$ and $q=\pm 1$ (so that it induces a Dehn filling on $\partial R_i\times \SSS^1$). The resulting structure is
$$ 1 + i c d(q\theta) + i d\log r \wedge d( p\theta - q\sigma) - \ldots $$
Act by the $A$-field $A:=c q d\log r$, as in \eqref{eq:A-field}, to get
$$ 1 + c q(d\log r + i d\theta) + i d\log r \wedge d( (p-c^2)\theta - q\sigma)-\ldots$$
Act by the $B$-field $B:=q d\theta\wedge d\sigma$, as in \eqref{eq:B-action}, to get
$$ 1 + c q(d\log r + i d\theta) + i (d\log r+i d\theta) \wedge d( (p-c^2)\theta - q\sigma).$$
We introduce the variable $z=re^{i\theta}$ so that the structure is written as
$$ 1 + cq \frac{dz}{z} + i \frac{dz}{z}\wedge d( (p-c^2)\theta - q\sigma).$$
Motivated by this, consider the structure on $T:=\D_{e^{a^2/2}}\times \SSS^1$ given by 
$$ \rho_T := z + cq dz - i q dz \wedge d\sigma + i \beta(r^2) (p-c^2) dz\wedge  d\theta,$$
where $\beta:[0,e^{a^2}]\to \R$ is a smooth function such that $\beta([0,x_0])=0$ and $\beta([x_1,b])=1$ for some $x_0<x_1<e^{b^2}$. The form $\rho_T$ defines a $B_3$-structure on the open solid torus $T$:
\begin{itemize}
    \item $\rho_T(0)=cq dz - i q dz \wedge d\sigma$ is pure, whereas for $z\neq 0$ we have that\\ $\rho_T=z \exp(cq dz/z - i q (dz/z) \wedge d\sigma + i \beta(r^2) (p-c^2) (dz/z)\wedge  d\theta)$ is also pure.
    \item $(\rho_T,\overline{\rho_T}) = -2icq^2 dz\wedge d\overline{z}\wedge d\sigma\neq 0$.
    \item $d\rho_T= -iq\partial_\sigma \cdot \rho_T $ (note that $\beta(r^2) (p-c^2) dz\wedge  d\theta$ is closed).
\end{itemize}

We consider the closed manifold $\widetilde{M}:=(M\setminus \overline{\D_b\times \SSS^1})\cup_\Psi T$, where we use $\Psi$ to glue along $R_0\times \SSS^1$ and $R_1\times \SSS^1$. The relation between the $B_3$-structures on $R_1\times \SSS^1$ is given, for $A$ and $B$ as above, by
\begin{equation}\label{eq:spinors-M-T}
    z e^A e^B \Psi^* \rho=\rho_T.
\end{equation}

We see next that this defines a twisted $B_3$-structure.

\begin{remark}
\label{remark:framing}
The underlying topological surgery is described by the coefficients $(p,\pm 1)$ and the isomorphism \eqref{eq:diffeo-N-D-S1}. This isomorphism induces a peripheral system:  two curves $m$ and $l$ that
are a basis for 
$H_1(\partial N,\mathbb Z)$, so that
$m$ is a meridian (that is, it bounds a disk in $N$).  Thus, through \eqref{eq:diffeo-N-D-S1},  $m$ corresponds to $\partial D^2_a$ and $l$ to the factor
$S^1$. For this peripheral system, $(p,\pm 1)$ surgery means that 
$p\,m\pm l$ is the filling meridian, which bounds a disk in the resulting manifold. 
Notice that different choices of the diffeomorphism  \eqref{eq:diffeo-N-D-S1} may have different peripheral systems, but the meridian does not change, 
so the $(p,\pm 1)$ surgery may be denoted by $(p',\pm 1)$ in a new peripheral system. Namely, $p$ may change by an integer but $\pm 1$ remains the same (in other words, the set of isotopy classes of orientation-preserving diffeomorphisms of the solid torus is $\mathbb Z$, generated by Dehn twists).
\end{remark}

\subsection{The twisting}\label{sec:the-twisting}

On $\widetilde{M}$ as above, consider the open sets corresponding to $M':=M\setminus \overline{\D_b\times \SSS^1}$ and $T$, for which we keep the same notation, and consider spinors $\rho_M$ and $\rho_T$, which on the intersection are related by $z e^A e^B\rho_M=\rho_T$ similarly as in \eqref{eq:spinors-M-T}. We consider some $A_M\in \Omega^1(M')$ and $A_T\in\Omega^1(T)$ such that $A=A_M-A_T$ on $M'\cap T$. From  \eqref{eq:eAeA'},
$$ e^A e^B = e^{-A_T}e^{A_M}e^{B-A_T\wedge A_M}=e^{-A_T}e^{A_M}e^{-C_T}e^{C_M}$$
for some $C_M\in \Omega^2(M)$ and $C_T\in\Omega^2(T)$ such that,  on $M'\cap T$, 
\begin{equation}\label{eq:BAC}
B-A_T\wedge A_M=C_M-C_T.    
\end{equation}
Note that $-dA_T=-dA_M$ defines a global closed $2$-form $F$, whereas, by differentiating \eqref{eq:BAC}, $-dC_M-dA_M\wedge A_M=-dC_T-dA_T\wedge A_T$ defines a global $3$-form $H$ satisfying $dH+F^2=0$.

Consider $e^{A_M}e^{C_M}\rho_M$ and $e^{A_T}e^{C_T}\rho_T$. By Lemma \ref{lem:twist-int} they are  twisted integrable for $H$ and $F$ as above. By \eqref{eq:spinors-M-T}, they differ on $M'\cap T$ by the non-vanishing function $z$, so they span a subbundle $K\subset \wedge^\bullet T^*_\C \widetilde{M}$ that gives a twisted $B_3$-structure (purity and real index zero follow from the corresponding properties for $\rho_M$ and $\rho_T$).

\begin{lemma}\label{lem:exact-F}
The surgery procedure just described introduces a twisting by an exact $F$.
\end{lemma}

\begin{proof}

We apply Stokes' Theorem: let $\xi$ be an arbitrary nonzero closed $1$-form,
\begin{equation}\label{eq:F-wedge-xi}
\begin{split} 
    \int_{\widetilde{M}} F\wedge \xi & {} = \int_{ \widetilde{M}\setminus T  } F\wedge \xi + \int_{ T  } F\wedge \xi  \\ 
     & {}  = -\int_{-\partial T}  A_M \wedge \xi - \int_{\partial T}  A_T \wedge \xi = \int_{\partial T} A\wedge \xi=0,
\end{split}
\end{equation}
where the last equality follows from the fact that $A=cq d\log r$ and $\partial T$ involves only the coordinates $\theta$ and $\sigma$. As $\widetilde{M}$ is closed, by Poincaré duality $H^1(\widetilde{M})^* \cong H^2(\widetilde{M})$, and we have that $F$ is exact.
\end{proof}

Say $F=d\Lambda$ in Lemma \ref{lem:exact-F}, for some $\Lambda\in\Omega^1(M)$, and consider the structure $e^{\Lambda}K$. By Lemma \ref{lem:twist-int}, $e^{\Lambda}K$ is twisted integrable for $F':=0$ and $H':=H + F\wedge \Lambda$.

\begin{lemma}\label{lem:non-exact-H}
    The form $H'$ as above is never exact, since it satisfies
    $$ \int_{\widetilde{M}} H' =   4\pi^2 q \neq 0  $$
\end{lemma}

\begin{proof}
    We apply Stokes' Theorem again, now on $H'$, using that $d\Lambda = - dA_M$ on $\widetilde{M}\setminus T$ and $d\Lambda= - dA_T$ on $T$. Note that
  \begin{multline*}  
\int_{T} (-dC_T -  dA_T\wedge A_T +  F\wedge \Lambda) = \int_{T} ( -dC_T  +  F\wedge (A_T+\Lambda) )\\  = \int_{T} d( - C_T  +  \Lambda \wedge (A_T+\Lambda) ) = \int_{\partial T} (- C_T + \Lambda\wedge A_T) 
\end{multline*}
 and analogously for $\widetilde{M}\setminus T$. Thus, 
\begin{align*}
    \int_{\widetilde{M}} H'
          = {} &  \int_{-\partial T}  (- C_M + \Lambda\wedge A_M)
     + \int_{\partial T} (- C_T + \Lambda\wedge A_T) \\ = {} &  \int_{\partial T} (B - A_M\wedge A_T - \Lambda\wedge A) =  \int_{\partial T} B + \int_{\partial T} ( (A_M-\Lambda) \wedge A ) \\ = {} &  \int_{\partial T} q d\theta\wedge d\sigma = 4\pi^2 q \neq 0, 
\end{align*}
since, just as in Lemma \ref{lem:exact-F}, the integral involving $A$ vanishes and the one involving $B$ can be computed directly, recalling that $\partial T=\SSS^1 \times \SSS^1$. Since $\widetilde{M}$ is closed, we have that $H'$ is never exact. \end{proof}


Lemma \ref{lem:non-exact-H} implies that, unlike for usual generalized complex structures \cite{cavalcanti-gualtieri:2007, cavalcanti-gualtieri:2009}, only one surgery will not be enough. Indeed, the situation is very different: in the usual case it suffices to have a $4$-manifold with vanishing $H^3(M,\mathbb{R})$, but for $B_3$-structures we always have that $H^3(M,\mathbb{R})\neq \{0\}$ and the $3$-form in the twisting is never closed. 

Based on the proofs of Lemmas \ref{lem:exact-F} and \ref{lem:non-exact-H}, we can perform several surgeries at the same time to obtain the following.

\begin{proposition}\label{prop:many-surgeries}
    Let $M$ be a manifold endowed with a $B_3$-structure with $2r$ disjoint cosymplectic neighbourhoods $\{ N_j \}$ equipped with diffeomorphisms as in \eqref{eq:diffeo-N-D-S1}. For any choice of $p_j\in \Z$,  the manifold $$M( (p_1,-1),\ldots,(p_{2r},+1) ),$$ resulting of performing $(p_j,(-1)^j)$-Dehn fillings, admits a  $B_3$-structure with $2r$ type-change curves more than $M$.
\end{proposition}

\begin{proof}
    Let $a_j>0$ be such that $N_j\cong D_{a_j}\times \SSS^1$ (as cosymplectic manifolds) and let $b_j$ satisfying $a_j>b_j>0$. We can perform  the surgeries of Section \ref{sec:p-1-surgery} simultaneously on each $N_j$ with maps $\Psi_r$ like in \eqref{eq:map-psi},         
    $$ M( (p_1,q_1),\ldots,(p_{2r},q_{2r})\} )\cong (M\setminus \cup_{j=1}^{2r} (\overline{D_{b_j}\times \SSS^1})) \cup_{\Psi_1} T_1 \ldots \cup_{\Psi_{2r}} T_{2r}, $$
    where $q_j=(-1)^j$. The resulting twisting is calculated as in Section \ref{sec:the-twisting}: the analogue of  \eqref{eq:F-wedge-xi} will have now $2r$ vanishing integrals, so the resulting $F$ is exact $F=d\Lambda$. We can act by $e^\Lambda$ so that only the twisting by $H$ remains. Its integral, as in the proof of Lemma \ref{lem:non-exact-H}, is $$4\pi^2 \sum_{j=1}^{2r} q_j = 4\pi^2 \sum_{j=1}^{2r} (-1)^j=0.$$
    So it is exact, $H=d Y$ for a $2$-form $Y$, and we can act by $e^Y$ to get a $B_3$-structure. The final claim follows from the fact that each surgery along $T_j$ introduces a type-change curve.
\end{proof}

By Remark \ref{remark:framing}, the diffeomorphism class of  $$M( (p_1,-1),\ldots,(p_{2r},+1) )$$ above depends on the choice of peripheral system for each cosymplectic neighbourhood.

\section{First examples}\label{sec:examples}

We apply now Proposition \ref{prop:many-surgeries} to generate new $B_3$-structures. We shall work with products $\Sigma\times S^1$ for $\Sigma$  a surface and do surgery on 
the $S^1$ factor, so that the peripheral system or choice of product structure
on the tubular neighbourhood is clear.

We start with an example in which we can compute everything explicitly. Consider the manifold $\SSS^2\times \SSS^1$ and the cosymplectic $B_3$-structure of Example \ref{ex:S2xS1}. Let $N$ and $S$ be the north and south poles of $\SSS^2$. We can perform $(1,1)$ and $(1,-1)$ surgeries along two non-intersecting cosymplectic neighbourhoods around $\{N\}\times \SSS^1$ and $\{S\}\times \SSS^1$. If we regard $\SSS^2\times \SSS^1$ as   $(D^2\times \SSS^1) \cup_\partial (D^2\times \SSS^1)$, the first surgery gives the manifold 
    $$ (D^2\times \SSS^1) \cup_\partial (S^1\times D^2) =\SSS^3,$$
whereas the second surgery gives again
$$ (S^1\times D^2) \cup_\partial (S^1\times D^2)= \SSS^1\times \SSS^2. $$
By choosing neighbourhoods around $N$ and $S$ to be hemispheres, the resulting $B_3$-structure is diffeomorphic to the one given on $ (\C\cup\{\infty\}) \times \SSS^1$ by 
\begin{equation}\label{eq:type-change-S2xS1}
\begin{cases}
			z+idz+idz\wedge d\theta, & \text{on } \C \times \SSS^1, \\
            \frac{1}{z}
            -id(\frac{1}{z})-id\frac{1}{z}\wedge d\theta, & \text{on } (\C^\times \cup \{\infty\} ) \times \SSS^1.
\end{cases}
\end{equation}

\begin{remark}
   By doing surgeries $(p_1,-1)$, $(p_2,1)$ on $\SSS^2\times \SSS^1$ we obtain the lens space $L(p_1-p_2,1)$ by \cite[Exercise 10.3.6]{martelli2022introduction}. 
\end{remark}


We replace now the sphere by a closed genus-$g$ oriented surface. We consider $\Sigma_g\times \SSS^1$ with the usual cosymplectic structure, and two points $P,Q\in\Sigma_g$ to perform a $(0,1)$ and a $(0,-1)$-surgery along disjoint neighbourhoods of $\{P\}\times \SSS^1$ and $\{Q\}\times \SSS^1$. The resulting manifold, $M:=(\Sigma_g\times \SSS^1)( (0,1),(0,-1) )$, has a $B_3$ structure with two type-changing curves.
A quick way to determine the resulting manifold is to compute its fundamental group.
Start with $\pi_1((\Sigma_g-\{P,Q\})\times S^1)\cong
F_{2g+1}\times \mathbb Z$, where $F_{2g+1}$ denotes the 
free group of rank $2g+1$. Then gluing the solid tori of surgery kills the factor $\mathbb Z$, and the resulting 
fundamental group is $F_{2g+1}$. By Thurston's 
geometrization
the resulting manifold must be  
 $\#^{2g+1} (\SSS^2\times \SSS^1)$. 
\begin{proposition}
    The manifold  $\#^{2g+1} (\SSS^2\times \SSS^1)$, for $g\geq 1$, is a $B_3$-manifold that is neither cosymplectic nor nacs.
\end{proposition}


\begin{proof}
	We have defined a $B_3$-structures on the manifold  $\#^{2g+1} (\SSS^2\times \SSS^1)$, for $g\geq 1$, which is a non-trivial connected sum, that is, not prime. 
	
	Cosymplectic manifolds are symplectic mapping tori (as mentioned in Remark \ref{rem:cosym-nacs}). A mapping torus is  prime, because either it is  $S^2\times S^1$ or its universal cover is $\mathbb R^3$, the universal cover of the fibre times the universal cover of $S^1$ \cite[Sec. 9.2.6]{martelli2022introduction}.
	On the other hand, by \cite[Thm. 2]{geiges-normal}, nacs manifolds are Seifert manifolds, but are not diffeomorphic to $\R P^3\#\R P^3$,  hence they are prime \cite[Sec. 10.3.12]{martelli2022introduction}. \end{proof}

We can apply more pairs of surgeries to get that the family $$\#^{2(g+s)-1} (\SSS^2\times \SSS^1),$$ for $g\geq 1$ and $s\geq 1$, is neither cosymplectic nor nacs but admits a $B_3$-structure whose type-change locus consists of $2s$ circles.  If we replace the 
$(0,\pm 1)$-surgeries by $(p,\pm 1)$-surgeries with all $p\neq 0$,  then the 
result is an orientable Seifert manifold
 \cite[Sec. 10.3]{martelli2022introduction} with an
 orientable base, hence it is nacs by Remark \ref{rem:cosym-nacs}. If any of the coefficients is
$(0,\pm 1)$, then the result can be seen not to be prime, see \cite[Cor. 10.3.43]{martelli2022introduction}.   A natural question now is whether we can find a geometric manifold with a $B_3$-structure that does admit  neither nacs nor cosymplectic structures. We actually prove much more.


\section{Every closed orientable $3$-manifold admits a $B_3$-structure}\label{sec:every-closed-B3}


We can and do assume our manifold to be connected. We will give a constructive proof organized as follows. Every closed orientable 3-manifold has an open book decomposition with connected binding, so it is a $(0,1)$-surgery on the mapping torus of a surface diffeomorphism (Section \ref{sec:open-book-mapping-torus}). We will prove that
it is the mapping torus of an area preserving diffeomorphism
(Section \ref{sec:invariant-area-form}) that fixes a disk
(Section \ref{sec:identitiy-around-fixed-point}), and we will apply surgery to a
pair of parallel curves in the mapping torus of this disk (Section \ref{sec:two-surgeries-B3-structure}).


\subsection{Every closed orientable $3$-manifold is obtained via surgery on a mapping torus with a fixed point}\label{sec:open-book-mapping-torus}

We review this fact for completeness. We recall the definition and properties of open book decompositions \cite{etnyre:06}.

\begin{definition}
    An open book decomposition of a closed $3$-manifold $M$ is given by a $1$-dimensional submanifold $B\subset M$, called the binding, and a fibration $\pi:M\setminus B\to \SSS^1$ such that $\pi^{-1}(\theta)$, called an open page, is the interior of a compact surface $\Sigma_\theta\subset M$ with $\partial \Sigma_\theta=B$.
\end{definition}

The following is a key result.

\begin{theorem}[\cite{alexander1923lemma,acuna:74,myers:78}]\label{thm:open-book}
    A closed oriented $3$-manifold admits an open book decomposition with {connected} binding.
\end{theorem}

All  open pages are diffeomorphic to the interior of a surface with boundary, say, $\Sigma$. The open book is diffeomorphic to the manifold obtained from the mapping torus $M_\phi$ of a diffeomorphism $\phi:\Sigma\to \Sigma$ that is the identity around $\partial\Sigma$, by collapsing the $\SSS^1$ factor in the product  
$\partial M_\phi\cong \partial\Sigma\times \SSS^1$.
Equivalently, assuming that the binding is connected, the initial $M$ is the result of attaching a solid torus along $\partial M_\phi$, so that the surgery meridian is  the  $\SSS^1$ factor (and the core of the solid torus is the binding). On the other hand, let 
$N$ denote the result of attaching a solid torus along  $\partial M_\phi$ but with surgery meridian
$\partial \Sigma$. The $\SSS^1$ factor of the solid torus is compatible with the bundle structure, and thus $N$ is the mapping torus of the diffeomorphism
obtained by attaching a disk to $\partial\Sigma$
and extending $\phi$ as the identity on this disk. 
Furthermore, we can pass from $N$ to the initial manifold $M$
by Dehn surgery because both $M$ and $N$ are the result of attaching a solid
torus $D^2\times S^1$ to $\partial M_\phi$. The coefficient of this surgery is $(0,1)$
in the standard peripheral system, because the roles of $\partial D^2$ and the factor $S^1$ are switched. Thus we have proved the following corollary.

\begin{corollary}
\label{cororllary:every-manifold-surgery-mapping-torus}
    Every $3$-manifold is obtained via $(0,1)$-surgery on a mapping torus (with a fixed point).    
\end{corollary}




\subsection{Existence of an invariant area form}\label{sec:invariant-area-form}

We work now on the mapping torus $M_\phi$ of a diffeomorphism (with a fixed point) $\phi:\Sigma\to \Sigma$, which is orientation preserving as $M_\phi$ is orientable. Take any area form $\omega$ on $\Sigma$. We have that $\phi^*\omega$ has the same area as $\omega$, so by Moser's argument, there exists a family of diffeomorphisms $\{f_t\}$  isotopic to the identity such that $f_t^*(t\phi^*\omega+(1-t)\omega)=\omega$, so, for $f:=f_1$, we have $f^* \phi^* \omega=\omega$.  The map $\phi\circ f$ is an area preserving diffeomorphism such that $M_{\phi \circ f}\cong M_\phi$ (since $\phi \circ f$ and $\phi$ are isotopic), but it may not have a fixed point any more. For this, we deform the gluing diffeomorphism~further.

Take any $x\in \Sigma$, consider a path to $(\phi\circ f)(x)$ and a connected and simply connected neighbourhood $V$ of this path. Cover the path with charts $U_0, \ldots, U_r\subset V$ so that $x_0:=x\in U_0$, $x_r:=(\phi\circ f)(x)\in U_r$ and choose points $x_j\in U_{j-1}\cap U_j$ for $0<j<r$. We use now a standard argument of Hamiltonian flows.

\begin{lemma}\label{lem:symplecto-transitive}
Symplectomorphisms act transitively on any two points of a chart. 
\end{lemma}

\begin{proof}
    Given two points of $y,z\in U\subset \mathbb{R}^n$ in a symplectic chart, consider two compact subsets $C\subset C'\subset U$ containing them and a bump function $\xi$ that is $1$ on $C$ and $0$ outside~$C'$.
    
    Consider the function $f(x):=\xi(x)\omega(z-y, x )$ with associated Hamiltonian vector field $X_f$, i.e., such that $df=\omega(X_f, \, )$. Since $x\mapsto \omega(z-y,x)$ is a linear map, on $C$ we have that $X_f=z-y$. The one-parameter subgroup integrating $X_f$ at time $1$ gives a symplectomorphism sending $y$ to $z$ which is the identity outside $C'$. 
\end{proof}

We apply Lemma \ref{lem:symplecto-transitive} to find symplectomorphisms $\phi_j$ sending $x_j$ to $x_{j+1}$ for $0\leq j<r$ and that are the identity outside $U_j$. The composition of these maps $g:=\phi_{r-1}\circ \ldots\circ \phi_0$ sends $x_0$ to $x_r$ and is the identity outside of $V$. It is isotopic to the identity (as it is a composition of maps isotopic to the identity) and we thus have that $M_{g^{-1}\circ \phi\circ f}=M_{\phi}$ with $g^{-1}\circ \phi\circ f$ having a fixed point and leaving the area form $\omega$ invariant.


\subsection{The gluing diffeomorphism can be chosen to be the identity around the fixed point} \label{sec:identitiy-around-fixed-point}


We use Dacorogona-Moser's Theorem:

\begin{theorem}[\cite{dacorogna-moser:1990}, {\cite[Thm. 10.11]{csato-dacorogna-kneuss}}]\label{thm:dacorogna-moser}
    Let $\Omega$ be a bounded connected open set in $\R^n$. Let $f,g\in\cCi(\overline{\Omega})$ such that i) $f,g > 0$ in $\overline{\Omega}$, ii) $\int_\Omega f=\int_\Omega g$ and iii) $\supp(f-g)\subset \Omega$. Then there exists $\phi\in \Diff(\overline{\Omega})$ such that $\phi^*g=f$ and $\supp(\phi-\Id)\subset \Omega$.
\end{theorem}

With it we can prove the following.

\begin{lemma}
\label{lemma:area_preserving_correction}
 The mapping torus of an area-preserving diffeomorphism with a fixed point $p$ is diffeomorphic to the mapping torus of an area-preserving diffeomorphism that is the identity around $p$.
\end{lemma}

\begin{proof}
   Let $h\colon\Sigma\to \Sigma$ denote the diffeomorphism (where $\Sigma$ is not necessarily a surface, although this will be the relevant case for us). Consider a chart around $p\in \Sigma$. Take an open set $D_2$ around $p$ and an open set $D_1$ whose closure is strictly contained in $D_2\cap h(D_2)$. Complete the family of subsets with $D_0$ and $D_3$ to have four nested open sets $$p\in D_0\subset D_1\subset D_2 \subset D_3 \subset \Sigma.$$  We define two concentric ring-like sets $R\subset \Omega$ by
\begin{align*}
    R & :=D_2 \setminus \overline{D_1},& \Omega&:=D_3\setminus \overline{D_0}.
\end{align*}

By the choice of $D_1$ and $D_2$, we have $h(\Sigma\setminus D_2)\cap D_1=\emptyset$, so we can define a diffeomorphism $\psi\colon\Sigma\to\Sigma$ such that
\begin{align*}
\psi_{|D_1}&=\Id_{|D_1},& \psi_{|M\setminus \overline{D_2}}&=h_{|M\setminus \overline{D_2}}.
\end{align*}


We now correct $\psi$ to make it area-preserving. For this we apply  Theorem \ref{thm:dacorogna-moser}, where we implicitly use the preimage via the chart. We consider $\Omega$ with the real-valued functions $f:=1$ and $g:=\Jac \psi$. The hypotheses are satisfied: i) $fg>0$ as $\psi$ is a diffeomorphism ($g$ does not vanish) and $g$ takes positive values, ii) $\int_\Omega f=\int_\Omega g$, 
as $\int_\Sigma g=\int_\Sigma \Jac \psi =\int_{\psi(\Sigma)} 1 = \int_\Sigma 1$, and $f=g$ on $\Sigma\setminus \Omega$, and iii) $\supp(f-g)\subset \Omega$ because $f=g$ on $\Sigma\setminus R$. Consequently, there exists $\phi\in \Diff(\Omega)$ such that $\phi^* \Jac \psi=1$ on $\Omega$ and $\supp(\phi-\Id)\subset \bar{R}\subset \Omega$. By the latter, we can extend the map $\phi$ by $\phi_{|\Sigma\setminus\Omega}:=\Id$ and thus consider the map $\phi^* \psi=\psi\circ \phi^{-1}$, which satisfies $\Jac \phi^*\psi = \phi^* \Jac \psi=1$. 

As $\phi$ is the identity apart from a disk, $\phi^*\psi$ is isotopic to $\psi$, so their mapping torus are diffeomorphic,
$$M_{h} \cong  M_{\psi} \cong M_{\phi\circ \psi},$$
and $\phi\circ \psi_{|D_0}=\Id_{|D_0}.$ \end{proof}

\subsection{Two surgeries to obtain a $B_3$-structure}
\label{sec:two-surgeries-B3-structure}
Let $M$ be a closed orientable $3$-manifold. By Section \ref{sec:open-book-mapping-torus} it can be obtained by $(0,1)$-surgery  from a mapping torus $M_\phi$, where, by Sections \ref{sec:invariant-area-form} and \ref{sec:identitiy-around-fixed-point}, $\phi:\Sigma\to \Sigma$ fixes an open set around $p$ and preserves an invariant form $\omega$ on $\Sigma$.

Consider the cosymplectic structure on $M_\phi$ given by $1  + id\theta + i\omega - d\theta\wedge \omega$. Take two fixed points $p,q$ in the disk fixed by the diffeomorphism $\phi$. Consider the corresponding curves $\SSS_p$ and $\SSS_q$. Perform now two surgeries from Section~\ref{sec:p-1-surgery}: a $(0,1)$ surgery on $\SSS_p$ and a $(1,- 1)$ surgery on $\SSS_q$ in $M_\phi$. After the first surgery, the resulting manifold is $M$,
by Corollary~\ref{cororllary:every-manifold-surgery-mapping-torus}. We claim that the second surgery still yields $M$.
To prove the claim, notice that $\SSS_p$ and  $\SSS_q$
are parallel in $M_\phi$ (their union is the boundary of an embedded annulus $[0,1]\times S^1$). Then after the first surgery, this annulus can be extended to a disk (by adding a meridian disk of the first surgery torus), so $\SSS_q$ bounds an embedded disk in $M$.
A neighbourhood of this disk in $M$ is  homeomorphic to a ball, in which $\SSS_q$ is the unknot. This means that after the second surgery we obtain  the connected sum of $M$ with the $(1,- 1)$ surgery on the unknot in $S^3$. The  $(1,-1)$ surgery
on the unknot in $S^3$ yields again $S^3$ (this surgery is the lens space $L(1,-1)$, by \cite[Sec.~10.1.2]{martelli2022introduction},
which is homeomorphic to $L(1,0)\cong S^3$ \cite[Thm.~10.1.12,  Prop.~10.1.6]{martelli2022introduction}). Hence 
the second surgery  does not change the manifold~$M$.


     By the choice of the coefficients $q$ and Proposition \ref{prop:many-surgeries}, we have thus proved:  
\begin{theorem}\label{thm:every-closed-B3}
    Any closed orientable $3$-manifold admits a $B_3$-structure with exactly two type-change curves.
\end{theorem}

A natural and interesting question arises: by analogy with the binding of an open-book decomposition in Theorem \ref{thm:open-book}, can we find examples of $B_3$-structures with a connected type-change locus? This and the study of the geometry and local invariants will be the subject of an upcoming work.

\bibliographystyle{alpha}
\bibliography{B3}

\bigskip 

\begin{small}

\noindent \textsc{J. Porti,  Universitat Aut\`onoma de Barcelona,
and Centre de Recerca Matem\`atica
08193 Barcelona, Spain }
\\
\noindent \textsf{joan.porti@uab.cat}

\medskip

\noindent \textsc{R. Rubio,  Universitat Aut\`onoma de Barcelona,
08193 Barcelona, Spain }
\\
\noindent \textsf{roberto.rubio@uab.es}

\end{small}

\end{document}